\documentclass[letterpaper, 10 pt, conference]{ieeeconf}  % Comment this line out
\IEEEoverridecommandlockouts                              % This command is only
\usepackage{graphics} % for pdf, bitmapped graphics files
\usepackage{epsfig} % for postscript graphics files
\usepackage{algorithm}
\usepackage{algorithmic}
\usepackage{amsmath} % assumes amsmath package installed
\usepackage{amssymb}  % assumes amsmath package installed
\usepackage{amsfonts}
\usepackage{tikz}
\usetikzlibrary{matrix,arrows}
\usepackage{authblk}

\newtheorem{theorem}{\bf Theorem}

\newtheorem{example}{\bf Example}

\newtheorem{remark}{\bf Remark}
\newtheorem{definition}{\bf Definition}

\newtheorem{proposition}{\bf Proposition}

\def\QED{~\rule[-1pt]{5pt}{5pt}\par\medskip}

\DeclareMathOperator*{\argmin}{arg\,min}

\newcommand{\farhad}[1]{\color{black}#1\color{black}}
\newcommand{\cedric}[1]{\color{black}#1\color{black}}

\title{\LARGE \bf
A Faithful Distributed Implementation of Dual Decomposition and Average Consensus Algorithms}

\author{Takashi Tanaka, Farhad Farokhi, C\'edric Langbort% <-this % stops a space
\thanks{T. Tanaka is with Laboratory for Information and Decision Systems, Massachusetts Institute of Technology, USA
        {\tt\small ttanaka@mit.edu}}%
\thanks{F. Farokhi is with ACCESS Linnaeus Center, KTH Royal Institute of Technology, Sweden
        {\tt\small farakhi@kth.se}}%
\thanks{C. Langbort is with Department of Aerospace Engineering, University of Illinois at Urbana-Champaign, USA
        {\tt\small langbort@illinois.edu}}%
}

\begin{document}

% The following packages can be found on http:\\www.ctan.org
\maketitle
\begin{abstract}
We consider large scale cost allocation problems and consensus seeking problems for multiple agents, in which agents are suggested to collaborate in a distributed algorithm to find a solution. If agents are strategic to minimize their own individual cost rather than the global social cost, they are endowed with an incentive not to follow the intended algorithm, unless the tax/subsidy mechanism is carefully designed. Inspired by the classical Vickrey-Clarke-Groves mechanism and more recent algorithmic mechanism design theory, we propose a tax mechanism that incentivises agents to faithfully implement the intended algorithm. In particular, a new notion of asymptotic incentive compatibility is introduced to characterize a desirable property of such class of mechanisms. The proposed class of tax mechanisms provides a sequence of mechanisms that gives agents a diminishing incentive to deviate from suggested algorithm.
\end{abstract}

\section{Introduction}

A common difficulty in large scale optimization problems that arise in social, infrastructural, and communication networks is the heavy computational load that cannot be handled by a single computer. A practical solution algorithm for such problems should thus be parallelizable so that the computational load can be distributed over the network agents. Another potential challenge in such problems comes from the fact that nobody has access to the entire data defining the optimization problem, because this information is often private and localized within distributed agents. Hence, it is desirable that a solution algorithm allows agents to implement the algorithm without knowing other agents' private information.

In this paper, we consider two distributed algorithms that are both attractive in the above sense: the \emph{dual decomposition} algorithm for cost allocation problems, and the \emph{average consensus} algorithms for a consensus seeking.
In the majority of the literature on these algorithms, it is assumed that the distributed agents are ``benevolent" and blindly follow the intended algorithm. However, if the society involves rational and strategic agents, it is more realistic to assume that they behave in more selfish manner in an effort to minimize their individual cost rather than the global social cost.
Hence in this paper, rather than \emph{assuming} that agents are collaborative, we consider a mechanism by which rational agents are \emph{incentivised} to follow the intended distributed algorithm.

Mechanism design theory (\cite{RefWorks:260,RefWorks:100,RefWorks:83} to mention a few) concerns the question \farhad{of } how the society can make a preferable decision with the presence of strategic agents.
The goal of a mechanism design is to suggest a social decision procedure (together with a tax/subsidy mechanism) that incentivises agents to follow the intended action.
However, the framework of the classical mechanism design theory is often insufficient to handle more complicated distributed algorithms.

It is relatively recent that the discipline of algorithmic mechanism design \cite{RefWorks:236} was recognized in computer science. Inspired by \cite{RefWorks:233}, we formulate a distributed mechanism design problem and suggest a tax mechanism that incentivises agents in a certain sense to follow the intended dual decomposition and the average consensus algorithm. Our tax mechanism can be seen as a generalization of the celebrated Vickrey-Clarke-Groves (VCG) mechanism.

By its nature, these two algorithms are asymptotic algorithms: if terminated at some finite step, only an approximation to the optimal solution is obtained.
It is recognized in \cite{RefWorks:108,RefWorks:235} that the VCG mechanism combined with such an approximated solution does not guarantee incentive compatibility in general. To circumvent this difficulty, we introduce a notion of \emph{asymptotic incentive compatibility} for a sequence of mechanisms that provides agents a diminishing incentive to deviate from the intended algorithm.

Finally, with the present study, we are able to relate our earlier study of real-time electricity pricing scheme \cite{RefWorks:261,RefWorks:237} to the mechanism design theory in more solid manner. Moreover, the framework of distributed mechanism design potentially allows us to implement the pricing scheme in \cite{RefWorks:261,RefWorks:237} in a smarter way: it allows distributed computations and allows agents not to disclose their private information.

\section{A Quick Review of Mechanism Design}
\subsection{Mechanisms}
Consider a society $E_N$ comprised of $N$ agents. Besides these agents, the society also has a government who makes a social decision $x\in X\subset \mathbb{R}^n$. Each agent $i=1,\cdots, N$ has private information $\theta_i\in\Theta_i$ called \emph{type}. When a social decision $x$ is made, the intrinsic cost $v_i(x;\theta_i)$ is charged to the $i$-th agent.
The government desires to make a social decision $x$ in such a way that the sum of individual costs is minimized. If the agents' types $\theta=(\theta_1,\cdots,\theta_N)\in \Theta_1\times\cdots\times\Theta_N =:\Theta$ are available to the government, the desired social decision is given by
$$
x(\theta)=\argmin_{x\in X} \sum_{i=1}^N v_i(x;\theta_i).
$$
A map $x:\Theta\rightarrow X$ is called a \emph{decision rule}. A decision rule is said to be \emph{efficient} if
$$
\sum_{i=1}^N v_i(x(\theta);\theta_i) \leq \sum_{i=1}^N v_i(x';\theta_i)
$$
for all $\theta\in \Theta$ and for all $x'\in X$.
Since $\theta$ is private in reality, each agent is asked to report his type (denoted by $\hat{\theta}_i$) to the government. In this process, we assume that agents are strategic in that they are allowed to make an unfaithful report (i.e., $\hat{\theta}_i\neq \theta_i$). Notice that agent $i$ has an incentive to do so if it will lead to an alternative social decision $x'$ that costs less for himself, i.e., $v_i(x';\theta_i)<v_i(x(\theta);\theta_i)$.

In order to encourage agents to be truthful, the government can introduce a monetary transfer function $t:\Theta\rightarrow \mathbb{R}^N$. This can be seen as a tax mechanism that the government imposes on each agent. We assume that the amount of tax imposed on each agent is determined by the government using the reported type $\hat{\theta}$:
$$
t(\hat{\theta})=\left(t_1(\hat{\theta}),
\cdots,t_N(\hat{\theta})\right).
$$
Overall, agent $i$'s net cost is given by
$$
u_i(\hat{\theta},\theta_i,x,t)=v_i(x(\hat{\theta});\theta_i)+t_i(\hat{\theta})
$$
and agents are expected to behave rationally to minimize this function.
The pair $f(\theta)=\left(x(\theta),t(\theta)\right)$, $f:\Theta\mapsto X \times \mathbb{R}^N$ is called a \emph{social choice function}.

The reporting process can be designed in an indirect manner using message functions $s_i:\Theta_i\rightarrow \Sigma_i$, where $\Sigma=\Sigma_1\times\cdots\times\Sigma_N$ is the space of messages. Upon receiving ``encoded" types $s(\theta)=\left(s_1(\theta_1),\cdots,s_N(\theta_N)\right)$, the government recovers the value of the social choice function using a ``decoding" function $g:\Sigma\rightarrow X \times \mathbb{R}^N$ such that $(g\circ s)(\theta)=f(\theta)$ for every $\theta\in\Theta$. A \emph{mechanism} is a triplet $M=(g,\Sigma, s)$ of an outcome function $g$, a space of messages $\Sigma$, and a particular encoding scheme $s:\Theta\rightarrow \Sigma$ such that $g\circ s=f$ \farhad{(see~Fig.~\ref{commdiag})}. One can think of $s$ as an intended encoding scheme that the government desires each agent to follow.  A particular case with $\Sigma=\Theta$, $g=f$, and $s=Id$ \farhad{(i.e., identity map) } is called a \emph{direct mechanism}, in which agents are asked to report their types $\theta_i$ to the government without encoding. A mechanism is said to be \emph{dominant strategy incentive compatible} if implementing the suggested encoding scheme $s=(s_1,\cdots,s_N)$ is a dominant strategy for each individual. If there exists such a dominant strategy incentive compatible mechanism $M=(g,\Sigma,s)$ such that $f=g\circ s$, $f$ is said to be implemented by $M$ in dominant strategies.
Since the government announces a mechanism first and the agents react to it, the government and the agents are also referred to as the leader and the followers in what follows.

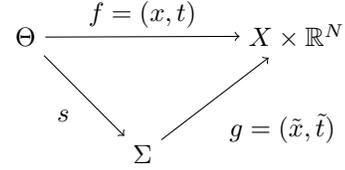
\begin{figure}[ht]
 \begin{center}
 \begin{tikzpicture}[descr/.style={fill=white,inner sep=2.5pt}]
 \matrix(m) [matrix of math nodes, row sep=3em,
 column sep=3em]
 { \Theta & & X \times \mathbb{R}^N \\ & \Sigma & \\ };
 \path[->]
 (m-1-1) edge node[auto] {$ f=(x,t) $} (m-1-3)
         edge node[below left=2pt] {$ s $} (m-2-2)
 (m-2-2) edge node[below right=2pt] {$ g=(\tilde{x},\tilde{t}) $} (m-1-3);
 \end{tikzpicture}
  \end{center}
\caption{\label{commdiag} A pair $f=(x,t)$ is called the social choice function while a pair $g=(\tilde{x},\tilde{t})$ is called the outcome function}
\end{figure}

\subsection{Fundamental results from mechanism theory}
If a mechanism $(g,\Sigma,s)$ implements a social choice function $f=(x,t)$ in dominant strategies, then a direct mechanism $(g\circ s, \Theta, Id)$ implements $f$ in dominant strategies as well (revelation principle). This suggests that a given social choice function is implemented by some mechanism $M$ if and only if it is implemented by some direct mechanism in dominant strategies.

If the government wishes to design a mechanism that implements a particular social decision rule $x$ without introducing monetary transfer function (i.e., $t\equiv 0$), then the decision rule has to be quite trivial (Gibbard-Satterthwaite theorem). This suggest that a tax (or subsidy) mechanism is almost necessary for the government to incentivise individuals to be truthful to make the right social decision $x$.

Let $x:\Theta\rightarrow X$ be an efficient social decision rule. It is elementary to prove, but valuable to realize that there exists a class of monetary transfer functions $t^{Groves}:\Theta\rightarrow \mathbb{R}^N$ such that a direct mechanism $(f,\Theta,Id)$ with the social choice function $f=(x,t^{Groves})$ is dominant strategy incentive compatible (\cite{RefWorks:83} for a great introduction). Such a class of monetary transfer functions are given by
\begin{subequations}
\label{groves}
\begin{align}
x^{Groves}(\hat{\theta})&=\argmin_{x\in X} \sum_{i=1}^N v_i(x;\hat{\theta}_i) \label{groves1}\\
t_i^{Groves}(\hat{\theta})&=k_i(\hat{\theta}_{-i})+\sum_{j \neq i} v_j(x^{Groves}(\hat{\theta});\hat{\theta}_j) \label{groves2}
\end{align}
\end{subequations}
for each $i=1,\cdots,N$, where $k_i:\prod_{j\neq i} \Theta_j \rightarrow \mathbb{R}$ is an arbitrary function that does not depend on $\hat{\theta}_i$. A mechanism obtained by the above scheme is referred to as a Groves' mechanism.
It is known that the form of monetary transfer function (\ref{groves2}) is not only sufficient but also necessary in an appropriate sense, in order for $x$ to be efficient and a mechanism $(f, \Theta, Id)$, $f=(x,t)$ is dominant strategy incentive compatible (Green-Laffont theorem).

\subsection{VCG mechanisms}
A particular choice of
$$k_i(\hat{\theta}_{-i})=-\min_{x\in X} \sum_{j\neq i} v_j(x;\hat{\theta}_j)
$$
yields a mechanism with some additional desirable properties. With this choice, the monetary transfer function becomes
\begin{equation}
\label{vcg}
t_i^{VCG}(\hat{\theta})=\sum_{j\neq i} v_j(x^{Groves}(\hat{\theta});\hat{\theta}_j)
-\min_{x\in X} \sum_{j\neq i} v_j(x;\hat{\theta}_j).
\end{equation}
Let $E_{N-i}$ denote the society excluding the $i$-th agent. The first term in (\ref{vcg}) corresponds to the total cost of $N-1$ agents (excluding $i$) when an efficient social decision is made for $E_N$. The second term in (\ref{vcg}) represents the minimum social cost achievable for $E_{N-i}$. Combined, (\ref{vcg}) means that the tax imposed on the $i$-th agent is the same amount as the marginal contribution of the $i$-th agent to the rest of the society. The Groves mechanism with tax policy (\ref{vcg}) is called the Vickrey-Clarke-Groves (VCG) mechanism, which is known to be advantageous from the viewpoint of \emph{budget balance} and \emph{individual rationality} \cite{RefWorks:100}. In order to compute (\ref{vcg}), the government needs to determine social decisions that minimize the cost for each of $E_N$ and $E_{N-i}, i=1,\cdots,N$.

\section{Distributed Mechanisms}
A large body of mechanism design theory focuses on direct mechanisms. Although this is partially justified by the revelation principle, clearly there are a number of practical situations in which ``indirect" mechanisms are preferable. For instance, indirect mechanisms allow distributed computations for large scale problems, while direct mechanisms require central computation by the government to determine the optimal social decision. Indirect mechanisms will also be advantageous for privacy preservation if they allow to find the optimal social decision without having individuals disclose their private information completely.

Previously, we considered $\Sigma$ as the space of messages and $g:\Sigma\rightarrow X\times \mathbb{R}^N$ was viewed as a decoding scheme. From this section on, we want to consider more general computational interactions between the leader and the followers than mere encoding-decoding interactions.
Specifically, we assume that the interaction between the leader and the followers occurs in multiple stages (indexed by $k=1,2,\cdots,K$). At each stage, the leader broadcasts his current computational output $y_L^k$ to the followers. Also, we assume that each follower transmits his current computational output $y_i^k$ directly to the leader (and possibly to the neighboring followers, depending on the communication topology) via secure channels. We assume that the leader can be modeled as a state-based computer with the internal state $z_L^k$, while the $i$-th follower can be modeled as a state-based computer with the internal state $z_i^k$. Given initial states $z_L^0,z_i^0$ and $y_L^0, y_i^0, i=1,\cdots,N$, the state evolves according to:
\begin{subequations}
\label{algall}
\begin{align}
z_i^k &= G_{i,\theta_i}^k(z_i^{k-1},y_L^{k-1}, \{y_j^{k-1}\}_{j\in N(i)}) \label{alg1}\\
y_i^k &= H_{i,\theta_i}^k(z_i^k) \label{alg2}\\
z_L^k &= G_L^k (z_L^{k-1}, y_1^k, \cdots, y_N^k) \label{alg3}\\
y_L^k &= H_L^k (z_L^k) \label{alg4}
\end{align}
\end{subequations}
for $k=1,2,\cdots,K$. In the above, $\{y_j^{k-1}\}_{j\in N(i)}$ represents the outputs of the neighboring followers, and hence we are considering a communication topology as in Fig.~\ref{figcommtop}.
Finally, we require that
$$
y_L^K=H_L^K(z_L^K)=(x,t)\in X\times \mathbb{R}^N
$$
which will be the value of the social choice.

A strategy of the $i$-th follower is the sequence of functions in (\ref{alg1}) and (\ref{alg2}):
$$
s_i(\theta_i)=\left\{(G_{i,\theta_i}^k, H_{i,\theta_i}^k): k=1,2,\cdots, K\right\}
$$
parametrized by his type $\theta_i$. On the other hand, the outcome function is defined by the sequence of functions in (\ref{alg3}) and (\ref{alg4}):
$$
g=\left\{(G_L^k, H_L^k): k=1,2,\cdots, K\right\}.
$$
\begin{figure}[tbp]
\begin{center}
\includegraphics[width=0.85\linewidth, bb=20 120 670 430]{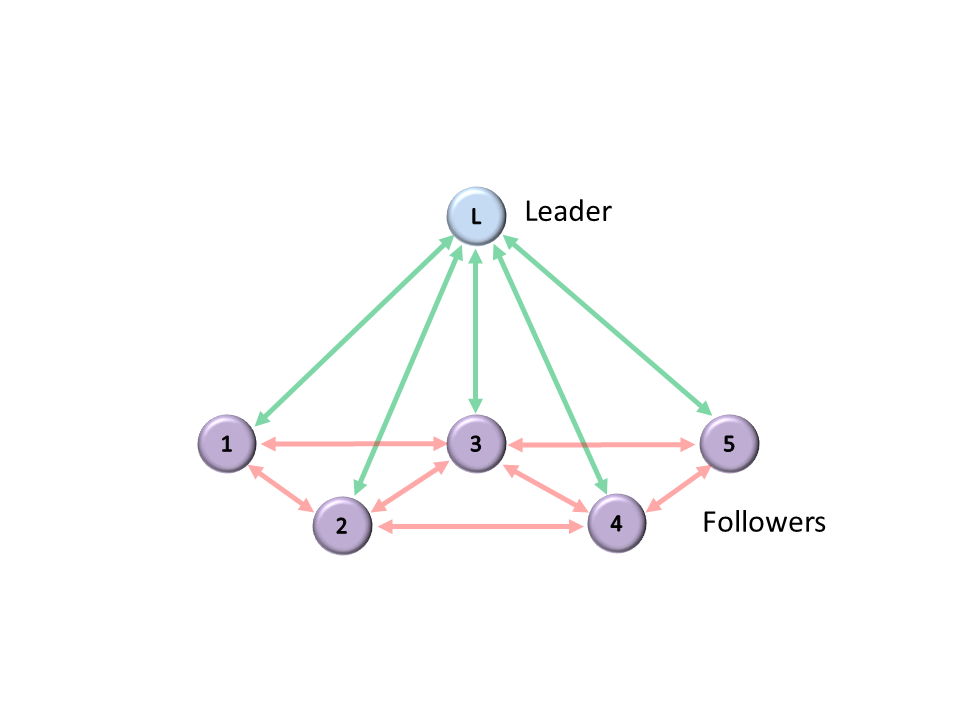}
\caption{Communication topology of distributed algorithms}
\label{figcommtop}
\end{center}
\end{figure}
We call $s_i(\cdot)$ a \emph{slave} algorithm, which can be seen as a map from $\Theta_i$ to the space $\Sigma_i$ of follower $i$'s strategies. On the other hand, $g$ is referred to as a \emph{master} algorithm.
Given an element in $\Sigma$, a master algorithm $g$ determines a social choice $(x,t)$. Hence, $g$ can be a map from $\Sigma$ to $X\times \mathbb{R}^N$ so that $(g\circ s)(\theta)=(x,t)$. This way, we can reuse the diagram of Fig.~\ref{commdiag} with a generalized interpretation.
To summarize:
\begin{itemize}
\item $\Sigma_i$ is the space of all possible strategies that could be taken by followers involved in a game. Notice that the state space description of a dynamical system is not unique. Hence, we define $\Sigma$ as a space of the equivalence classes of slave algorithms that have the same input-output behavior. A concrete description of the strategy space varies depending on the problem of interest.
\item $s_i:\Theta_i\rightarrow \Sigma_i$ is a mapping that determines a strategy of the $i$-th follower with type $\theta_i$. Each of these mappings for $i=1,\cdots,N$ is referred to as a slave algorithm.
\item A mapping $g:\Sigma\rightarrow X \times \mathbb{R}^N$, $g(s)=(\tilde{x}(s),\tilde{t}(s))$ is referred to as the master algorithm. A composition $a=g\circ s: \Theta\rightarrow \mathbb{R}^N$ is called a distributed algorithm.
\end{itemize}
With this re-interpretation, a mechanism $(g,\Sigma,s)$ is defined as a triplet of a master algorithm $g$, followers' strategy space $\Sigma$, and a \emph{suggested} slave algorithm $s$.

The notion of dominant strategy incentive compatibility is often too strong and hard to achieve. Hence, it is common to employ a weaker notion of incentive compatibility.
\begin{definition}
\label{defic}
A mechanism $(g,\Sigma,s)$ \emph{implements} a social choice function $f$ in \emph{ex-post Nash equilibria} if
\begin{itemize}
\item[(1).] $g\circ s=f$
\item[(2).]$\forall i, \forall \hat{s}_i\in\Sigma_i, \forall \theta\in \Theta$,
\begin{align*}
&v_i\left(\tilde{x}_i\circ (s_i(\theta_i),s_{-i}(\theta_{-i}));\theta_i\right)
+\tilde{t}_i\circ (s_i(\theta_i),s_{-i}(\theta_{-i})) \\
&\leq v_i\left(\tilde{x}_i\circ (\hat{s}_i,s_{-i}(\theta_{-i}));\theta_i\right)
+\tilde{t}_i\circ (\hat{s}_i,s_{-i}(\theta_{-i}))
\end{align*}
\end{itemize}
\end{definition}
where $s=(s_i,s_{-i})$. In this case, the mechanism is also said to be \emph{incentive compatible}.
% \begin{remark}
% As in \cite{RefWorks:233}, we assume that players are self-interested but will benevolently follow the suggested strategy unless there exists another strictly preferable strategy.
% \end{remark}

\section{Dual decomposition}
Consider the following cost allocation problem.
\begin{subequations}
 \label{originalprob}
\begin{align}
\min &\;\; \sum_{i=1}^N v_i(x_i;\theta_i) \label{originalprob1}\\
\text{s.t.} &\;\; Rx=c  \label{originalprob2}
\end{align}
\end{subequations}
A vector $x=[x_1^T;\cdots;x_N^T]^T\in X$ is a concatenation of the social decision variables, and the domain $X$ is defined by an affine constraint $Rx=c$. We assume that $R=[R_1 \; \cdots \; R_N]$ is full row rank.

Let $L(x,p)= \sum_{i=1}^N v_i(x_i;\theta_i)+p^T(Rx-c)$ be the Lagrangian of (\ref{originalprob}). The dual function is given by
$$
g(p)=\inf_{x} L(x,p)=\sum_{i=1}^N \inf_{x_i}\left(v_i(x_i;\theta)+p^TR_ix_i\right),
$$
and the dual problem is
$
\sup_p g(p).
$
The primal-dual optimal solution $(x^*,p^*)$ constitutes a saddle point of $L(x,p)$, and assuming strict convexity of $v_i(\cdot;\theta)$, the saddle point value $L^*$ corresponds to the optimal value of (\ref{originalprob}).
The following iteration is guaranteed to converge to $(x^*,p^*)$
\begin{subequations}
\label{ahu}
\begin{align}
&\hat{x}_i^k=\argmin_{\hat{x}_i} \left( v_i(\hat{x}_i^{k-1},\theta_i)+{p^{k-1}}^TR_i\hat{x}_i^{k-1} \right) \label{ahu1} \\
&p^k=p^{k-1}+ \gamma(R\hat{x}^k-c) \label{ahu2}
\end{align}
\end{subequations}
if the step size $\gamma$ is chosen to be sufficiently small.
At each step of the above iteration, we have an upper and lower bound for the optimal value of (\ref{originalprob}). A lower bound $\underbar{b}$ can be computed using current $\hat{x}$ and $p$ by
\begin{align*}
\underbar{b}&:= \sum_{i=1}^N \left(v_i(\hat{x}_i; \theta_i)+p^TR_i\hat{x}_i\right) \\
&= \inf_x L(x,p) \leq \sup_p \inf_x L(x,p)=L^*.
\end{align*}
Although the equality constraint (\ref{originalprob2}) may not be satisfied by the current $\hat{x}$, the nearest feasible point $x$ from $\hat{x}$ is given by
$$
x=\hat{x}-R^T(RR^T)^{-1}(R\hat{x}-c).
$$
Using this $x$, an upper bound $\bar{b}$ of the optimal value of (\ref{originalprob}) can be computed by
$$
\bar{b}:= \sum_{i=1}^N v_i(x_i;\theta_i)\geq L^*.
$$
Practically, one terminates the iterative procedure (\ref{ahu}) once the observed tolerance $\bar{b}-\underbar{b}$ is sufficiently small.

Notice that the above algorithm has an attractive form for a distributed implementation since (\ref{ahu1}) can be executed by the $i$-th follower and (\ref{ahu2}) can be executed by the leader. Namely (\ref{ahu1}) and (\ref{ahu2}) is in the form of (\ref{algall}) (See (\ref{statespace}) below).
By increasing the number of iterations, the intended social decision can be approximated with an arbitrary accuracy, provided that the followers faithfully implement (\ref{ahu1}). In what follows, we consider what kind of side payment mechanism $t_i$ suffices to incentivise followers to be faithful.

\subsubsection{Pure competitive market}
In a pure competitive market, every follower \emph{believes} that the price of a commodity is a given constant that cannot be manipulated by his sole action (i.e, followers are \emph{price-takers}). Such an assumption is employed in the standard t\^atonnement process which, after $K$ iterations, charges $t_i={p^K}^TR_ix_i$ on each follower. In this model, the dual variable $p$ can be naturally understood as the price of commodities, which defines followers' net cost
$$
u_i=v_i(x_i;\theta_i)+p^TR_ix_i.
$$
With the belief that the price is locally constant, following (\ref{ahu1}) is a rational choice for a price-taking follower.

\subsubsection{Oligopoly}
Many realistic markets are oligopoly, in which there exists a stakeholder who knows that his actions give certain effects on the market \cite{RefWorks:238}. In this case, a stakeholder might be better off by deviating from (\ref{ahu1}) and taking an alternative strategy.

\begin{example}
Consider the following simple problem:
\begin{align*}
& \min \;\; \frac{1}{2}x_1^2+\frac{1}{2}x_2^2 \\
& \text{ s.t. } \;\; x_1+x_2=1.
\end{align*}
The iteration (\ref{ahu}) leads to the primal-dual optimal solution $(x_1^*,x_2^*, p^*)=(1/2,1/2,-1/2)$. The corresponding optimal value is $L^*=1/2$ and agents' augmented costs are $u_1^*=u_2^*=-1/8$. To demonstrate that following (\ref{ahu1}) is not necessarily a rational strategy for the agents when the tax mechanism is $t_i=px_i$, suppose that player $1$ is a stakeholder (a quantity leader, \cite{RefWorks:238}) who knows how $p$ and $x_2$ react to his action $x_1$. This leads to a game with two agents, who are trying to minimize $u_i(x_i,p)=1/2 x_i^2+px_i$, $i=1,2$, and a market who is trying to maximize $L(x_1,x_2,p)=1/2x_1^2+1/2x_2^2+p(x_1+x_2-1)$. Assuming that the agent $1$ is the leader in the Stackelberg game and others are followers, reaction curves are given by $x_2(x_1)=1-x_1$ and $p(x_1)=x_1-1$.
Hence, suppose that agent $1$ takes a strategy to follow the following update rule instead of (\ref{ahu}):
\begin{equation}
\label{algleader}
\hat{x}_1^k=\hat{x}_1^{k-1}-\gamma \frac{d}{d\hat{x}_1}u_1(\hat{x}_1,p(\hat{x}_1)).
\end{equation}
With other players remaining to follow (\ref{ahu}), the new dynamics leads to a Stackelberg  equilibrium  $(x_1^*,x_2^*, p^*)=(1/3,2/3,-2/3)$. Notice that player $1$ achieves a smaller augmented cost $u_1^*=-1/6$ at the new equilibrium, even though the distributed algorithm as a whole clearly failed to find the solution to the original resource allocation problem. In this sense, the strategy of following (\ref{ahu}) by no means constitutes a Nash equilibrium.
\end{example}

\section{Faithful implementation of dual decomposition}

We are going to introduce a monetary transfer function $t$ that incentivises followers to implement (\ref{ahu}). For a sufficiently large $K$, a social decision
$$
x:=\hat{x}^K-R^T(RR^{T})^{-1}(R\hat{x}^K-c)
$$
approximates the optimal solution of (\ref{originalprob}) with an arbitrary accuracy. A natural choice of $t_i$ inspired by the Groves mechanism is then
\begin{equation}
\label{grovestax}
t_i:=\sum_{j\neq i} v_i(x_i;\theta_i).
\end{equation}
Throughout the algorithm, we require the followers to report not only $x_i^k$ but also the value $v_i(x_i^k;\theta)$ so that the leader is able to perform a convergence analysis and compute $t_i$. The proposed distributed algorithm is summarized in Algorithm \ref{distalg}. Note that the algorithm is parametrized by $n\in \mathbb{N}$, which guarantees the accuracy $1/n$ of the final result.
\begin{algorithm}
\caption{Distributed mechanism $M_n=(g^n,\Sigma^n,s^n)$}
\begin{algorithmic}
\label{distalg}
\REQUIRE $\text{Problem data, step size $\gamma$ }$
\ENSURE $\text{Social outputs } \tilde{x}, \tilde{t}_1,\cdots,\tilde{t}_N$
\STATE $\text{(F) Initialize } \hat{x}_1, \cdots, \hat{x}_N;$
\STATE $\text{(L) Initialize and broadcast } p \text{ and } x;$
\WHILE{$\bar{b}-\underbar{b} > 1/n$}
\STATE $\text{(F) Compute and report } v_i=v_i(x_i; \theta_i);$
\STATE $\text{(F) Find } \hat{x}_i=\argmin_{x_i}(v_i(x_i;\theta_i)+p^TR_ix_i)$
\STATE $\text{ and report } \hat{x}_i \text{ and } \hat{v}_i=v_i(\hat{x}_i;\theta_i);$
\STATE $\text{(L) Compute upper bound } \bar{b}=\sum_{i=1}^N v_i;$
\STATE $\text{(L) Compute lower bound } \underbar{b}=\sum_{i=1}^N(\hat{v}_i+p\hat{x}_i);$
\STATE $\text{(L) Compute constraint violation } e=R\hat{x}-c;$
\STATE $\text{(L) Update dual variable } p:=p+\gamma e;$
\STATE $\text{(L) Find nearest feasible point } x=\hat{x}-R(RR^T)^{-1}e;$
\STATE $\text{(L) Broadcast } p  \text{ and } x;$
\ENDWHILE
\STATE $\text{(L) Determine social decision } \tilde{x}=x;$
\STATE $\text{(L) Determine tax values }\tilde{t}_i=\sum_{j\neq i} v_i;$
\end{algorithmic}
\end{algorithm}

%&&&&&&&&&&&&&&
% \begin{algorithm}
% \caption{Distributed mechanism $M_{\color{red}\infty}=(g^{\color{red}\infty},\Sigma^{\color{red}\infty},s^{\color{red}\infty})$}
% \begin{algorithmic}
% \label{distalg}
% \REQUIRE $\text{Problem data, step size $\gamma$ }$
% \ENSURE $\text{Social outputs } \tilde{x}, \tilde{t}_1,\cdots,\tilde{t}_N$
% \STATE $\text{(F) Initialize } \hat{x}_1, \cdots, \hat{x}_N;$
% \STATE $\text{(L) Initialize and broadcast } p \text{ and } x;$
% \WHILE{$\bar{b}-\underbar{b} > 1/{\color{red}n} $}
% \STATE $\text{(F) Compute and report } v_i=v_i(x_i; \theta_i);$
% \STATE $\text{(F) Find } \hat{x}_i=\argmin_{x_i}(v_i(x_i;\theta_i)+p^TR_ix_i)$
% \STATE $\text{ and report } \hat{x}_i \text{ and } \hat{v}_i=v_i(\hat{x}_i;\theta_i);$
% \STATE $\text{(L) Compute upper bound } \bar{b}=\sum_{i=1}^N v_i;$
% \STATE $\text{(L) Compute lower bound } \underbar{b}=\sum_{i=1}^N(\hat{v}_i+p\hat{x}_i);$
% \STATE $\text{(L) Compute constraint violation } e=R\hat{x}-c;$
% \STATE $\text{(L) Update dual variable } p:=p+\gamma e;$
% \STATE $\text{(L) Find nearest feasible point } x=\hat{x}-R(RR^T)^{-1}e;$
% \STATE $\text{(L) Broadcast } p  \text{ and } x;$
% \ENDWHILE
% \STATE $\text{(L) Determine social decision } \tilde{x}=x;$
% \STATE $\text{(L) Determine tax values }\tilde{t}_i=\sum_{j\neq i} v_i;$
% \end{algorithmic}
% \end{algorithm}
%&&&&&&&&&&&&&&

Note that Algorithm \ref{distalg} suggests followers' strategies $s^n(\cdot)$ and a social output function $g^n=(\tilde{x}^n,\tilde{t}^n)$. One possible state space description of the follower's strategy
$$
s_i(\theta)=\left\{(G_{i,\theta_i}^k, H_{i,\theta_i}^k): k=1,2,\cdots,K(n)\right\}.
$$
 is obtained by considering $y_L=(p; x)$ as the leader's output and
\begin{subequations}
\label{statespace}
\begin{align}
&\underbrace{\left[\begin{array}{c}
v_i \\ \hat{x}_i
\end{array}\right]}_{z_i^k}=
\underbrace{\left[\begin{array}{c}
v_i(x_i) \\ \argmin_{x_i} \left(v_i(x_i)+p^TR_ix_i\right)
\end{array}\right]}_{G_{i,\theta_i}^k(z_i^{k-1},y_L^{k-1})} \\
&\underbrace{\left[\begin{array}{c}
v_i \\ \hat{x}_i \\ \hat{v}_i
\end{array}\right]}_{y_i^k}=
\underbrace{\left[\begin{array}{c}
v_i \\ \hat{x}_i \\ v_i(\hat{x}_i;\theta_i)
\end{array}\right]}_{H_{i,\theta_i}^k(z_i^k)}.
\end{align}
\end{subequations}
The number of steps $K(n)$ is not known a priori but is finite. The strategy space $\Sigma_i^n$ for the $i$-th follower is the space of causal mappings from $y_L^0,\cdots,y_L^{K(n)}$ to $y_i^0,\cdots,y_i^{K(n)}$. This way Algorithm \ref{distalg} defines a sequence of mechanisms $\{M_n\}_{n\in \mathbb{N}}$, $M_n=(g^n,\Sigma^n,s^n)$.

Notice that the payment $t$ obtained in Algorithm \ref{distalg} is only an approximation of the correct VCG payment because, as per (\ref{grovestax}), it is computed using the final value of $x$, as opposed to the optimal social decision. Hence generally it does not guarantee incentive compatibility in the sense of Definition \ref{defic}. Such a fragile aspect of the VCG mechanism is considered in \cite{RefWorks:108}.
\begin{definition}
\label{defaic}
A sequence of mechanisms $\{M_n\}_{n\in\mathbb{N}}$, $M_n=(g^n,\Sigma^n,s^n)$ is said to \emph{asymptotically implement} a social choice function $f$ in ex-post Nash equilibria if for every $\delta_1>0,\delta_2>0$, there exists $N\in\mathbb{N}$ such that for any $n\geq N$,
\begin{itemize}
\item[(1).] $\|g^n\circ s^n-f\|<\delta_1$
\item[(2).] $\forall i, \forall \hat{s}_i^n\in\Sigma_i^n, \forall \theta\in \Theta$,
\begin{align*}
&\hspace{-.2in}v_i\left(\tilde{x}_i^n\circ (s_i^n(\theta_i),s_{-i}^n(\theta_{-i})));\theta_i\right)
+\tilde{t}_i^n\circ (s_i^n(\theta_i),s_{-i}^n(\theta_{-i})) \\
&\hspace{-.2in}< v_i\left(\tilde{x}_i^n\circ (\hat{s}_i^n,s_{-i}^n(\theta_{-i}));\theta_i\right)
+\tilde{t}_i^n\circ (\hat{s}_i^n,s_{-i}^n(\theta_{-i})) +\delta_2.
\end{align*}
\end{itemize}
In this case, $\{M_n\}_{n\in\mathbb{N}}$ is said to be \emph{asymptotically incentive compatible}.
\end{definition}
\begin{remark}
For a fixed $n\in \mathbb{N}$, $M_n$ is not incentive compatible. However, as $n\rightarrow \infty$, $M_n$ provides every follower a diminishing incentive to deviate from the suggested slave algorithm.
\end{remark}
\begin{theorem} \label{tho:1}
Assume that $v_i(\cdot;\theta_i), i=1,2,\cdots,N$ are strictly convex for every $\theta_i\in\Theta_i$, and define a social choice function $f=(x,t)$ by
\begin{align*}
& x:\Theta\rightarrow X, x(\theta)=\argmin_{x_1,\cdots,x_N} \sum_{i=1}^N v_i(x_i;\theta_i) \\
& t:\Theta\rightarrow \mathbb{R}^n, t_i(\theta)=\sum_{j\neq i} v_j(x_j(\theta);\theta_j).
\end{align*}
Then the sequence of mechanisms $\{M_n\}_{n\in\mathbb{N}}$  provided in Algorithm \ref{distalg} asymptotically implements $f$ in ex-post Nash equilibria.
\end{theorem}
\begin{proof}
By the convergence property of the dual decomposition algorithm, we have
$$
\tilde{x}^n \circ s^n(\theta)\rightarrow x(\theta) \text{ as } n\rightarrow \infty.
$$
As a result, by the continuity of $v_i$, $\forall i=1,\cdots, N$,
\begin{align*}
\tilde{t}_i^n \circ s^n(\theta)&=\sum_{j\neq i} v_j(\tilde{x}_j^n\circ s^n(\theta);\theta_j) \\
&\rightarrow \sum_{j\neq i} v_j(x_j(\theta);\theta_j)=t_i(\theta)
\end{align*}
as $n\rightarrow \infty$. This proves the first condition of Definition \ref{defaic}. To prove the second condition, suppose that there exist a sequence of strategies $\{\hat{s}_i^n\}_{n\in\mathbb{N}}$, $\delta_2>0$, and a subsequence $\{n_l\}$ in $\mathbb{N}$ such that
\begin{align}
&v_i\left(\tilde{x}_i^{n_l}\circ (s_i^{n_l}(\theta_i),s_{-i}^{n_l}\theta_{-i}));\theta_i\right)
+\tilde{t}_i^{n_l}\circ (s_i^{n_l}(\theta_i),s_{-i}^{n_l}(\theta_{-i})) \nonumber \\
&\geq v_i\left( \underbrace{\tilde{x}_i^{n_l}\circ (\hat{s}_i^{n_l},s_{-i}^{n_l}(\theta_{-i}))}_{\hat{x}};\theta_i\right)
+\underbrace{\tilde{t}_i^{n_l}\circ (\hat{s}_i^{n_l},s_{-i}^{n_l}(\theta_{-i}))}_{\hat{t}_i} +\delta_2 \label{contradiction}
\end{align}
for all $l\in\mathbb{N}$. Notice that
$$
\hat{x}:=\tilde{x}^{n_l}\circ (\hat{s}_i^{n_l},s_{-i}^{n_l}(\theta_{-i}))\in X
$$
is a feasible point in the original optimization problem, and
$$
\hat{t}_i:=\tilde{t}_i^{n_l}\circ (\hat{s}_i^{n_l},s_{-i}^{n_l}(\theta_{-i}))=\sum_{j\neq i} v_j(\hat{x}_j;\theta_j).
$$
Hence
\begin{align*}
(\text{RHS of (\ref{contradiction})}) &= v_i(\hat{x}_i;\theta_i)+\sum_{j\neq i} v_j(\hat{x}_j;\theta_j)+\delta_2 \\
&=\sum_{i=1}^N v_i(\hat{x}_i;\theta_i)+\delta_2 \geq L^*+\delta_2.
\end{align*}
On the other hand, by definition of Algorithm \ref{distalg}, it is guaranteed that
$$
(\text{LHS of (\ref{contradiction})}) = \sum_{i=1}^N v(\tilde{x}_i^{n_l};\theta_i)=\bar{b}^{n_l}
$$
with $\bar{b}^{n_l}-\underbar{b}^{n_l} \leq 1/n_l$. Since $\underbar{b}^{n_l}\leq L^* \leq \bar{b}^{n_l}$,
$(\text{LHS of (\ref{contradiction})})\leq L^*+1/n_l$.
Thus we have shown that
$$
L^*+\delta_2\leq (\text{RHS of (\ref{contradiction})}) \leq (\text{LHS of (\ref{contradiction})}) \leq L^*+1/n_l.
$$
However, since it is possible to take a sufficiently large $l$ so that $\delta_2>1/n_l$, the above inequality lead to a contradiction.
\end{proof}

\section{Dynamic dual decomposition}
In our earlier study \cite{RefWorks:261,RefWorks:237}, we have proposed a real-time electricity pricing scheme that incentivises strategic consumers/generators over the power grid to implement the socially optimal control action. We have assumed that there is no private information so that leader is able to compute the socially optimal control. The motivation of introducing a payment mechanism in our scenario was not to induce followers a faithful information revelation as in the classical mechanism design problems but to induce them to take the intended control actions.
Our payment mechanism was strongly inspired by the VCG mechanism, but due to this difference, the connection to the classical mechanism design setting was not transparent.

Notice that the formulation of the distributed mechanism introduced in this paper contains both the classical mechanism design problem and the pricing scheme \cite{RefWorks:261,RefWorks:237} as special cases. The classical mechanism design problem corresponds to the single step case ($K=1$) in (\ref{algall}), while the pricing scheme in \cite{RefWorks:261,RefWorks:237} corresponds to the case where there is no private information ($\Theta$ is a trivial singleton set).

Moreover, the distributed mechanism for dual decomposition algorithm considered in this paper suggests a better implementation of the real-time pricing scheme.
In our earlier study, the leader (the central computer) needs to solve a large scale optimal control problem in the centralized manner. Using the idea of dynamic dual decomposition (e.g., \cite{RefWorks:239}) combined with the VCG-like tax mechanism (\ref{grovestax}), computation can be faithfully parallelized. Furthermore, since the current framework is built on non-singleton type space $\Theta$, it allows strategic power generators/consumers to have private information.
More details will be explored in our future work.

\section{Faithful Average Consensus Implementation}
\cedric{As an application of the approach presented above, } we \cedric{now } consider average consensus seeking using dual decomposition when dealing with strategic agents. Let an undirected graph $\mathcal{G}=(\{1,\dots,N\},\mathcal{E})$, with vertex set $\{1,\dots,N\}$ and edge set $\mathcal{E}$, be given to illustrate the communication links between the agents (see Fig.~\ref{figcommtop}). Following~\cite{rabbat2005generalized}, we can achieve the average consensus through solving the optimization problem
\begin{subequations} \label{optimization:1}
\begin{equation}
\min_{x\in\mathbb{R}^N} \hspace{.1in} \sum_{i=1}^N (x_i-\theta_i)^2, \end{equation}
\begin{equation}
\mathrm{s.t.} \hspace{.1in} x_i=x_j, \forall (i,j)\in\mathcal{E},
\end{equation}
\end{subequations}
where $x_i\in\mathbb{R}$ denotes the decision variable of agent~$i$, $1\leq i\leq N$, and $\theta_i\in\Theta_i\subseteq\mathbb{R}$ is its type. Note that our assumption of considering scalar consensus problem is only in place to simplify the presentation and the results can be readily extended to higher dimensional cases using the same line of reasoning. Let us introduce the incidence matrix of $\mathcal{G}$. To do so, we need to assign arbitrary directions to the edges of $\mathcal{G}$. It is important to note that the underlying graph (specifically, the communication graph) is still an undirected graph. Let us define the incidence matrix $B(\mathcal{G})\in\{-1,0,+1\}^{N\times |\mathcal{E}|}$ so that $b_{ij}(\mathcal{G})=1$ if the edge $e_j\in\mathcal{E}$ leaves vertex $i$, $b_{ij}(\mathcal{G})=-1$ if the edge $e_j\in\mathcal{E}$ enters vertex $i$, and $b_{ij}(\mathcal{G})=0$ otherwise. In the rest of the section, we assume that $\mathcal{G}$ is a tree. Using the incidence matrix, we can rewrite the optimization problem in~\eqref{optimization:1} as
\begin{subequations} \label{optimization:2}
\begin{equation}
\min_{x\in\mathbb{R}^N} \hspace{.1in} \sum_{i=1}^N (x_i-\theta_i)^2, \end{equation}
\begin{equation}
\mathrm{s.t.} \hspace{.1in} Rx=0,
\end{equation}
\end{subequations}
where $R=B(\mathcal{G})^\top$. Clearly, the optimization problem~\eqref{optimization:2} is of the form discussed in~(4) when substituting $v_i(x_i;\theta_i)=(x_i-\theta_i)^2$ for all $1\leq i\leq N$. Noting that this optimization problem satisfies the Slater's condition, the duality gap is indeed zero and we can solve the problem using the dual decomposition~\cite[p.\,226]{boyd2004convex}. \cedric{ As a result, Algorithm~\ref{distalg} can be used to handle situation where, unlike in the classical literature (e.g.,~\cite{rabbat2005generalized,schizas2007consensus,ghadimi2011accelerated}), the agents engaged in the averaging process are strategic. However, we present two other algorithms which, unlike Algorithm~\ref{distalg}, allow for direct communication between the followers and therefore, can be considered more desirable. } \farhad{Before stating the results, let us define the sequence of mechanisms $\{M'_n\}_{n\in\mathbb{N}}$, where each mechanism $M'_n$ is introduced in Algorithm~\ref{alg:2}. Furthermore, note that Algorithm~\ref{alg:2} suggests followers' strategies $s^n(\cdot)$ and a social output function $g^n=(\tilde{x}^n, \tilde{t}^n)$. }

%The problem of seeking the average consensus using the dual decomposition or other distributed optimization methods have been studied in the literature~\cite{rabbat2005generalized,schizas2007consensus,ghadimi2011accelerated}
%However, these studies implicitly assume that the agents are not strategic and that they are collaboratively trying to calculate the average of $\{\theta_i\}_{i=1}^N$. Contrary to those studies, here, we consider the case where the agents are strategic and apply the results of the previous section to asymptotically implement the average consensus. Let us define the sequence of mechanisms $\{M'_n\}_{n\in\mathbb{N}}$, where each mechanism $M'_n$ is introduced in Algorithm~\ref{alg:2}. \farhad{Similarly, Algorithm~\ref{alg:2} suggests followers' strategies $s^n(\cdot)$ and a social output function $g^n=(\tilde{x}^n, \tilde{t}^n)$. } Through the following proposition, we prove that $\{M'_n\}_{n\in\mathbb{N}}$ asymptotically implements the average consensus.

\begin{algorithm}[t]
\caption{\label{alg:2} Distributed mechanism $M'_n=(g^n,\Sigma^n,t^n)$ for asymptotically implementing the average consensus.}
\begin{algorithmic}
\REQUIRE Problem data, step size $\gamma$
\ENSURE Social outputs $\tilde{x}$, $\tilde{t}_1,\dots,\tilde{t}_N$
\STATE (F) Initialize $x_1,\dots,x_N$;
\STATE (L) Initialize and broadcast $p$;
\REPEAT
\STATE (F) Each agent solves $x_i=\argmin_{z\in\mathbb{R}} (z-\theta_i)^2+p^TR_iz$ and transmit it to the leader;
\STATE (F) Each agent calculates $v_i=(x_i-\theta_i)^2$ and transmit it to the leader;
\STATE (F) Update dual variables $p_\ell= p_\ell+\gamma (x_i-x_j)$ for all edges $e_\ell=(i,j)\in\mathcal{E}$;
\STATE (L) Compute $\|Rx\|_2$;
\UNTIL{$\|Rx\|_2\leq 1/n$}
\STATE (L) Determine the social decision $\tilde{x}=x$;
\STATE (L) Determine the tax values $\tilde{t}_i=\sum_{j\neq i}v_j$;
\end{algorithmic}
\end{algorithm}

\begin{proposition} \label{prop:1} Define a social choice function $f=(x,t)$ by
\begin{subequations} \label{eqn:social:function}
\begin{equation}
x:\Theta\rightarrow X, x(\theta)=\left(\frac{1}{N}\sum_{i=1}^N \theta_i\right)\mathbf{1},
\end{equation}
\begin{equation}
t:\Theta\rightarrow \mathbb{R}^n, t_i(\theta)=\sum_{j\neq i} (x_j(\theta)-\theta_j)^2,
\end{equation}
\end{subequations}
where $\mathbf{1}$ denotes the vector of all ones in $\mathbb{R}^N$. Then the sequence of mechanisms $\{M'_n\}_{n\in\mathbb{N}}$ provided in Algorithm~\ref{alg:2} asymptotically implements $f$ in ex-post Nash equilibria.
\end{proposition}

\begin{proof} The proof follows the same line of reasoning as in the proof of Theorem~\ref{tho:1}.
\end{proof}

\begin{remark} As we have described in Algorithm~\ref{alg:2}, the agents need to solve $x_i=\argmin_{z\in\mathbb{R}} (z-\theta_i)^2+p^TR_iz$. This optimization problem has a explicit solution $x_i=\theta_i-0.5p^TR_i$. Therefore, at each iteration, the agents only need to apply a simple linear update rule, calculate the new cost, and send these information to the leader.
\end{remark}

\begin{algorithm}[t]
\caption{\label{alg:3} Distributed mechanism $M''_n=(g^n,\Sigma^n,t^n)$ for asymptotically implementing the average consensus.}
\begin{algorithmic}
\REQUIRE Problem data
\ENSURE Social outputs $\tilde{x}$, $\tilde{t}_1,\dots,\tilde{t}_N$
\STATE \farhad{(L) Set $\alpha\in(0,1/d_{\mathrm{max}})$ (where $d_{\mathrm{max}}$ } \cedric{denotes } \farhad{the maximum degree of the vertices in $\mathcal{G}$) and broadcast it;}
\STATE (F) Initialize $\farhad{z}_i(0)=\theta_i$ for each $1\leq i\leq N$;
\STATE (F) Initialize $\tau=0$;
\REPEAT
\STATE (F) Increase iteration number $\tau$ by one;
\STATE (F) Each agent calculates $\farhad{z}_i(\tau)=\farhad{z}_i(\tau-1)+\alpha \sum_{j\in\mathcal{N}_i} \linebreak[4](\farhad{z}_j(\tau-1) -\farhad{z}_i(\tau-1))$, where $\mathcal{N}_i$ is the set of all neighbors of vertex $i$ in~$\mathcal{G}$, and transmit it to the leader;
\STATE (F) Each agent computes $v_i=(\farhad{z}_i(\tau)-\theta_i)^2$ and transmit it to the leader;
\STATE (L) Calculate $\|R\farhad{z}(\tau)\|_2$;
\UNTIL{$\|R\farhad{z}(\tau)\|_2\leq 1/n$}
\STATE (L) Determine the social decision $\tilde{x}=\farhad{z}(\tau)$;
\STATE (L) Determine the tax values $\tilde{t}_i=\sum_{j\neq i}v_j$;
\end{algorithmic}
\end{algorithm}

\begin{figure*}[t]
\centering
\begin{equation} \label{eqn:longequation:1}
\bigg|\big[(\tilde{x}_i^{n}\circ(s_i^{n}(\theta_i),s_{-i}^{n} (\theta_{-i}))-\farhad{\theta_i})^2+\tilde{t}_i^{n}\circ (s_i^{n}(\theta_i),s_{-i}^{n}(\theta_{-i})) \big]- \big[(x_i(\theta)-\theta_i)^2+\sum_{j\neq i} (x_j(\theta)-\theta_j)^2\big]\bigg|\leq \delta_2/2,
\end{equation}
\vspace{-.07in}
\hrule
\begin{equation} \label{eqn:longequation:2}
\bigg|\big[(\tilde{x}_i^{n}\circ(\hat{s}_i^{n}(\theta_i),s_{-i}^{n} (\theta_{-i}))-\farhad{\theta_i})^2+\tilde{t}_i^{n}\circ (\hat{s}_i^{n}(\theta_i),s_{-i}^{n}(\theta_{-i}))\big]- \big[ (\hat{x}_i(\theta)-\theta_i)^2+\sum_{j\neq i} (\hat{x}_j(\theta)-\theta_j)^2\big]\bigg|\leq \delta_2/2,
\end{equation}
\vspace{-.07in}
\hrule
\end{figure*}

Note that $\{M'_n\}_{n\in\mathbb{N}}$ is not the only sequence of mechanisms that asymptotically implements the average consensus. In order to show this, we define the sequence of mechanisms $\{M''_n\}_{n\in\mathbb{N}}$ using Algorithm~\ref{alg:3} and show that this sequence indeed asymptotically implements the average consensus. \farhad{Again, Algorithm~\ref{alg:3} suggests followers' strategies $s^n(\cdot)$ and a social output function $g^n=(\tilde{x}^n, \tilde{t}^n)$. }

\begin{proposition} Define a social choice function $f=(x,t)$ by~\eqref{eqn:social:function}. Then the sequence of mechanisms $\{M''_n\}_{n\in\mathbb{N}}$ provided in Algorithm~\ref{alg:3} asymptotically implements $f$ in ex-post Nash equilibria.
\end{proposition}

\begin{proof} Following~\cite{xiao2004fast}, \farhad{since $\mathcal{G}$ is a tree } and $\alpha\in(0,1/d_{\mathrm{max}})$, we have \farhad{$\lim_{\tau\rightarrow \infty}z(\tau)=(N^{-1}\sum_{i=1}^N \theta_i)\mathbf{1}.$ } %Now, notice that as $n$ tends to infinity, $K(n)$ (i.e., the number of the iteration to achieve $\|R\farhad{z(K(n))}\|\leq 1/n$) also goes to infinity. 
Thus, we clearly get \farhad{$\lim_{n\rightarrow \infty}\tilde{x}^n\circ s^n(\theta)=x(\theta)$. } Considering the continuity of the cost functions, we can also recover \farhad{$\lim_{n\rightarrow \infty}\tilde{t}_i^n\circ s^n(\theta) =\sum_{j\neq i} \left(x_j(\theta)-\theta_j\right)^2.$ }  Evidently, for any $\delta_2>0$, there exists $n_1\in\mathbb{N}$ such that~\eqref{eqn:longequation:1} holds true for all $n\geq n_1$.
Therefore, 
\begin{equation} \label{eqn:proof:3}
\begin{split}
(\tilde{x}_i^{n}&\circ(s_i^{n}(\theta_i),s_{-i}^{n} (\theta_{-i}))-\farhad{\theta_i})^2\\&+\tilde{t}_i^{n}\circ (s_i^{n}(\theta_i),s_{-i}^{n}(\theta_{-i}))-\delta_2/2 \\ &\hspace{.4in}\leq (x_i(\theta)-\theta_i)^2+\sum_{j\neq i} (x_j(\theta)-\theta_j)^2.
\end{split}
\end{equation}
Now, assume that there exists an index $i$ such that agent~$i$ follows $\{\hat{s}_i^n\}_{n\in\mathbb{N}}$. Clearly, by the construction of Algorithm~\ref{alg:3}, we have
\farhad{$\lim_{n\rightarrow \infty} R\tilde{x}^{n}\circ(\hat{s}_i^{n}(\theta_i),s_{-i}^{n} (\theta_{-i}))=0.$ }
Therefore, because of the fact that $\farhad{z}(\tau)=(I-\alpha R^\top R)\farhad{z}(\tau-1)$, we know that the limit $\lim_{n\rightarrow \infty}\tilde{x}^{n}\circ(\hat{s}_i^{n}(\theta_i),s_{-i}^{n} (\theta_{-i}))$ indeed exists. Let us use the notation \farhad{$\hat{x}(\theta)=\lim_{n\rightarrow \infty}\tilde{x}^{n}\circ(\hat{s}_i^{n}(\theta_i),s_{-i}^{n} (\theta_{-i}))$}. Because of the continuity of the cost functions, for any $\delta_2>0$, there exists $n_2\in\mathbb{N}$ such that~\eqref{eqn:longequation:2} holds true for all $n\geq n_2$. Therefore,
\begin{equation} \label{eqn:proof:4}
\begin{split}
(\hat{x}_i(\theta)-&\theta_i)^2+\sum_{j\neq i} (\hat{x}_j(\theta)-\theta_j)^2\\ & \leq (\tilde{x}_i^{n}\circ(\hat{s}_i^{n}(\theta_i),s_{-i}^{n} (\theta_{-i}))-\farhad{\theta_i})^2\\&\hspace{.4in}+\tilde{t}_i^{n} \circ(\hat{s}_i^{n}(\theta_i),s_{-i}^{n}(\theta_{-i}))+\delta_2/2.
\end{split}
\end{equation}
\farhad{Furthermore, because $\hat{x}(\theta)$ is a feasible point and $x(\theta)$ is the global solution of~\eqref{optimization:2} (see~\cite{rabbat2005generalized}), the following inequality holds
\begin{equation} \label{eqn:proof:globalminimizer}
\begin{split}
(x_i(\theta)-&\theta_i)^2+\sum_{j\neq i} (x_j(\theta)-\theta_j)^2\\&\leq (\hat{x}_i(\theta)-\theta_i )^2+\sum_{j\neq i} (\hat{x}_j(\theta)-\theta_j)^2.
\end{split}
\end{equation}
Finally, } combining~\eqref{eqn:proof:3}, \eqref{eqn:proof:4}, \farhad{and~\eqref{eqn:proof:globalminimizer} } results in
\begin{equation*}
\begin{split}
&(\tilde{x}_i^{n}\circ(s_i^{n}(\theta_i),s_{-i}^{n} (\theta_{-i}))-\farhad{\theta_i})^2\\&\hspace{.7in}+\tilde{t}_i^{n} \circ (s_i^{n}(\theta_i),s_{-i}^{n}(\theta_{-i}))-\delta_2/2
\\ &\leq (\tilde{x}_i^{n}\circ(\hat{s}_i^{n}(\theta_i),s_{-i}^{n} (\theta_{-i}))-\farhad{\theta_i})^2
\\&\hspace{.7in}+\tilde{t}_i^{n}\circ (\hat{s}_i^{n}(\theta_i),s_{-i}^{n}(\theta_{-i}))+\delta_2/2,
\end{split}
\end{equation*}
for $n\geq \max(n_1,n_2)$. This concludes the proof.
\end{proof}
\begin{remark}
Algorithms~\ref{alg:2} and~\ref{alg:3} allow direct communications between followers, while Algorithm~\ref{distalg} involves only leader-follower communications. Also, unlike Algorithm~\ref{distalg}, the social decision $\tilde{x}^n\circ s^n$ as an output of Algorithms~\ref{alg:2} or~\ref{alg:3} may not be feasible (feasibility holds only at the limit, i.e., $\lim_{n\rightarrow \infty}\tilde{x}^n\circ s^n \in X$). Nevertheless, the notion of asymptotic incentive compatibility (Definition \ref{defaic}) is still applicable.
\end{remark}

\farhad{
\section{Conclusions}
We presented a framework for faithful implementation of dual-decomposition algorithms as well as average consensus seeking algorithms in a network of strategic agents. We introduced the notion of asymptotic incentive compatibility for a sequence of mechanisms, that is, this sequence provides every follower
a diminishing incentive to deviate from the suggested slave
algorithm. We proposed a tax mechanism, inspired by the classical Vickrey--Clarke--Groves mechanisms, to asymptotically implements a social choice function in ex-post Nash equilibria.
}

\bibliographystyle{ieeetr}
\bibliography{DualDecomposition}

\end{document}